\documentclass[11pt]{amsart}

\usepackage{verbatim,url}
\usepackage{amsmath}
\usepackage{enumerate,supertabular}
\usepackage{amssymb}
\usepackage{color}
\numberwithin{equation}{section}

\newtheorem{Theorem}{Theorem}[section]
\newtheorem{Lemma}[Theorem]{Lemma}
\newtheorem{Corollary}[Theorem]{Corollary}

\setbox0=\hbox{$+$}
\newdimen\plusheight
\plusheight=\ht 0
\newcommand\+{\;\lower\plusheight\hbox{$+$}\;}

\setbox0=\hbox{$-$}
\newdimen\minusheight
\minusheight=\ht0
\renewcommand\-{\;\lower\minusheight\hbox{$-$}\;}

\setbox0=\hbox{$\cdots$}
\newdimen\cdotsheight
\cdotsheight=\plusheight
\newcommand\cds{\lower\cdotsheight\hbox{$\cdots$}}

\newcommand{\bin}[2]{\begin{pmatrix} #1 \\ #2 \end{pmatrix}}

\begin{document}

\title[Supercongruences]{Supercongruences satisfied by coefficients of
$_2F_1$ hypergeometric series}
\author[Chan,  Kontogeorgis, Krattenthaler, Osburn]
{Heng Huat Chan,  Aristides Kontogeorgis, Christian Krattenthaler
and Robert Osburn}
\address{Department of Mathematics, National
University of Singapore, 2~Science Drive 2, Singapore 117543;
Max-Planck-Institut f\"ur Mathematik, Vivatsgasse 7, D-53111, Bonn, Germany}%
   \email{matchh@nus.edu.sg}

\address{
Department of Mathematics, University of the Aegean, 83200
Karlovassi, Samos, Greece} \email{aristides.kontogeorgis@gmail.com}

\address{Fakult\"at f\"ur Mathematik, Universit\"at Wien, Nordbergstrasze 15,
A-1090 Vienna, Austria\newline \leavevmode\indent {\it WWW}: {\tt
http://www.mat.univie.ac.at/\~{}kratt}}

\address{School of Mathematical Sciences, University College Dublin, Belfield, Dublin 4,
Ireland} \email{robert.osburn@ucd.ie}


\date{\today}

\dedicatory{Dedicated to Paulo Ribenboim on the occasion of his 80th
birthday}

\begin{abstract}
Recently, Chan, Cooper and Sica conjectured two congruences for
coefficients of classical $_2F_1$ hypergeometric series which also
arise from power series expansions of modular forms in terms of
modular functions. We prove these two congruences using
combinatorial properties of the coefficients.
\end{abstract}

\maketitle

\section{Introduction}

The sequence
$$\alpha_n=\sum_{k=0}^n \bin{n}{k}^2\bin{n+k}{k}^2,$$
introduced by R.~Ap\'ery \cite{AperAA} in his proof of the
irrationality of $\zeta(3)$, has many interesting arithmetical
properties.
For example, F.~Beukers \cite[p.~276]{Beukers-87} showed that
$\alpha_n$ arises from the power series expansion of a modular form
of weight 2 in terms of a modular
function.%
\footnote{Beukers gave the modular form in terms of Lambert series.
The product form can be found in \cite{KZ}.}
   More precisely, if $q=e^{2\pi i\tau}$ with $\text{Im}\,
\tau>0$,
\begin{gather*}
\eta(\tau) = q^{1/24}\prod_{n=1}^\infty (1-q^n),\\
Z(\tau) =
\frac{\bigl(\eta(2\tau)\eta(3\tau)\bigr)^7}{\bigl(\eta(\tau)\eta(6\tau)\bigr)^5}
\quad\text{and}\quad  X(\tau)  =
\left(\frac{\eta(\tau)\eta(6\tau)}{\eta(2\tau)\eta(3\tau)}\right)^{12}
   ,
\end{gather*}
   then
   \begin{equation}\label{eq1.1}
   Z(\tau)=\sum_{n=0}^\infty \alpha_nX^n(\tau).
\end{equation}
   Other properties of $\alpha_n$ were soon discovered
by S.~Chowla, J.~Cowles and M.~Cowles \cite{CCC}. They showed that
for all primes $p>3$,
$$\alpha_p\equiv \alpha_1 \pmod{p^3}.$$
Subsequently, I.~M.~Gessel \cite{Gessel} showed that, for all
positive integers $n$ and primes $p>3$,
\begin{equation}\label{eq1.2}
   \alpha_{np}\equiv \alpha_n\pmod{p^3}.
\end{equation}

Recently, an analogue of Ap\'ery numbers was found. The
corresponding sequence is formed by the Domb numbers \cite{CCL},
defined by
$$\beta_n=(-1)^n\sum_{k=0}^n\bin{n}{k}^2\bin{2k}{k}\bin{2(n-k)}{n-k}.$$
It can be shown (see \cite[(4.14)]{CCL}) that if
$$\mathcal Z(\tau) =
\frac{\bigl(\eta(\tau)\eta(3\tau)\bigr)^4}{\bigl(\eta(2\tau)\eta(6\tau)\bigr)^2}\quad
\text{and}\quad \mathcal X(\tau) =
\left(\frac{\eta(2\tau)\eta(6\tau)}{\eta(\tau)\eta(3\tau)}\right)^6,$$
then
\begin{equation}\label{eq1.3}
   \mathcal Z(\tau) =\sum_{n=0}^\infty \beta_n\mathcal X^n(\tau).
\end{equation}

In \cite{CCS}, H.~H.~Chan, S.~Cooper and F.~Sica showed, using
Gessel's idea, that
\begin{equation}\label{eq1.4}
\beta_{np}\equiv \beta_n\pmod{p^3}.
\end{equation}
The similarities between \eqref{eq1.1} and \eqref{eq1.3}, as well as
between \eqref{eq1.2} and \eqref{eq1.4}, indicated that perhaps
sequences arising from power series expansions of modular forms of
weight 2 in terms of modular functions may have properties similar
to \eqref{eq1.2} and \eqref{eq1.4}. Motivated by this idea, Chan,
Cooper and Sica constructed seven sequences $a_n$ from
$\eta$-quotients, analogues of theta functions and various modular
functions,
and they conjectured that, under certain conditions on the primes
$p$, these seven sequences satisfy congruences of the type
\begin{equation}\label{eq1.5}
      a_{np}\equiv a_n\pmod{p^r},
\end{equation}%
with $r= 1$, $2$, or $3$. Unfortunately, these conjectures do not
follow immediately from Gessel's method, and therefore new methods
have to be devised. The purpose of this  note is to give an
elementary approach to proving two of these conjectures.

\begin{Theorem}\label{T1.1}
Let $(a)_n=(a)(a+1)(a+2)\cdots (a+n-1)$.
\begin{enumerate}[\rm (a)]

\item For $p\equiv 1\pmod{4}$ and
$$s_n =64^n\frac{\left(\frac{1}{4}\right)_n^2}{(1)_n^2},$$
we have
\begin{equation}\label{eq1.6}
s_{np}\equiv s_n \pmod{p^2}.
\end{equation}

\item For $p\equiv 1\pmod{6}$ and
$$t_n =
108^n\frac{\left(\frac{1}{6}\right)_n\left(\frac{1}{3}\right)_n}{(1)_n^2},$$
we have
\begin{equation}\label{eq1.6.1} t_{np}\equiv t_n
\pmod{p^2}.
\end{equation}
\end{enumerate}
\end{Theorem}

The proof of \eqref{eq1.6} will be given in Sections~2 to 4. The
proof of \eqref{eq1.6.1} will be given in Section~5. Some parts of
the proof of \eqref{eq1.6.1} will only be sketched as they are
similar to that of \eqref{eq1.6}.

We conclude this introduction by indicating the analogues of
\eqref{eq1.1} and \eqref{eq1.3}.
%

Let
$$Z_2=\sum_{m=-\infty}^\infty\sum_{n=-\infty}^\infty
q^{m^2+n^2} \quad\text{and}\quad X_2=\frac{\eta^{12}(2\tau)}
{\displaystyle
Z_2^6 }.$$ Then the $s_n$'s are obtained from the expansion
$$Z_2=\sum_{n=0}^\infty s_nX_2^n.$$
Incidentally, the coefficients $s_n$ can be obtained from the
coefficients\break
$\left(\frac{1}{4}\right)_n\left(\frac{3}{4}\right)_n/(1)_n^2$
studied by S.~Ramanujan via a special case of Kummer's
transformation
$$\, _2F_1\left(\frac{1}{4},\frac{3}{4};1;x\right)=\frac{1}{\root{4}\of{1-x}}\, _2F_1\left(\frac{1}{4},\frac{1}{4};1;\frac{x}{x-1}\right),$$
where $\, _2F_1(a,b;c;z)$ is the classical Gau\ss ian hypergeometric
series.

Let $$Z_3=\sum_{m=-\infty}^\infty\sum_{n=-\infty}^\infty
q^{m^2+mn+n^2}\quad\text{and}\quad
X_3=\frac{\eta^6(\tau)\eta^6(\tau)} {
Z_3^6}.$$ Then the $t_n$'s are obtained from the expansion
$$Z_3=\sum_{n=0}^\infty t_nX_3^n.$$
The series associated with the coefficients $t_n$ were studied in
\cite{Borweins-Garvan} and \cite{Chan-Chua-Sole}, and these
coefficients are related to the coefficients
$\left(\frac{1}{3}\right)_n\left(\frac{2}{3}\right)_n/(1)_n^2$
studied by Ramanujan and the Borweins by means of the transformation
formula
$$\, _2F_1\left(\frac{1}{3},\frac{2}{3};1;x\right)= \,
_2F_1\left(\frac{1}{3},\frac{1}{6};1;4x(1-x)\right).$$

We remark here that, using \eqref{eq3.4}, it is immediate (see
\eqref{eq3.3} and \eqref{eq5.0.5}) that, if
$u_n=64^n\left(\frac{1}{4}\right)_n\left(\frac{3}{4}\right)_n/(1)_n^2$
and
$v_n=27^n\left(\frac{1}{3}\right)_n\left(\frac{2}{3}\right)_n/(1)_n^2$,
then
$$u_p\equiv u_1\pmod{p^2}\quad\text{and}\quad v_p\equiv v_1\pmod{p^2}.$$
Although it is not clear how one can deduce the corresponding
congruences for $s_p$ and $t_p$ from congruences satisfied by $u_p$
and $v_p$ using the $\, _2F_1$ transformation formulas, our proof of
Theorem~\ref{T1.1} is clearly motivated by these relations.



\section{A Lemma for the proof of \eqref{eq1.6}}

In this section, we establish a simple lemma which is interesting in
its own right.

\begin{Lemma}\label{L2.0} For positive integer $n$ and prime
$p\equiv 1\pmod{4}$,
\begin{equation}\label{eq2.1}
\left(\frac{3}{4}\right)_p\equiv
3\left(\frac{1}{4}\right)_p\pmod{p^3}.
\end{equation}

\end{Lemma}

\begin{proof}
By isolating the terms involving multiples of $p$ on both sides of
\eqref{eq2.1}, we find that it suffices to prove the congruence

\begin{equation}
\label{eq2.2}
\prod_{k=0}^{\frac{3p-7}{4}}\left(\frac{3}{4}+k\right)\prod_{k=\frac{3p+1}{4}}^{p-1}
\left(\frac{3}{4}+k\right) \equiv
\prod_{k=0}^{\frac{p-5}{4}}\left(\frac{1}{4}+k\right)\prod_{k=\frac{p+3}{4}}^{p-1}
\left(\frac{1}{4}+k\right) \pmod{p^2}.\end{equation} Let the product
on the left-hand side be $L(p)$ and the product on the right-hand
side be $R(p)$. We group some of the terms in $L(p)$ in pairs as
follows:
$$\left(\frac{3}{4}+\frac{3p-3}{4}-k\right)
\left(\frac{3}{4}+\frac{3p-3}{4}+k\right)$$ for $$1\leq k\leq
\frac{p-1}{4}.$$ We then conclude that
$$L(p) \equiv \prod_{k=1}^{\frac{p-1}{4}} (-k^2)\prod_{k=0}^{\frac{p-3}{2}}\left(\frac{3}{4}+k\right)\pmod{p^2}.$$

Similarly, for  $$1\leq k\leq \frac{p-1}{4},$$  we perform the
following pairing of some of the terms in the product in $R(p)$:
$$\left(\frac{1}{4}+\frac{p-1}{4}-k\right)
\left(\frac{1}{4}+\frac{p-1}{4}+k\right).$$
Hence we have
$$R(p) \equiv \prod_{k=1}^{\frac{p-1}{4}} (-k^2)\prod_{k=\frac{p+1}{2}}^{p-1}\left(\frac{1}{4}+k\right) \pmod{p^2}.$$

It now remains to verify that
\begin{equation}\label{eq2.3}\prod_{k=0}^{\frac{p-3}{2}}
\left(\frac{3}{4}+k\right)\equiv
\prod_{k=\frac{p+1}{2}}^{p-1}\left(\frac{1}{4}+k\right)\pmod{p^2}.\end{equation}

Denoting the left-hand side of \eqref{eq2.3} by $l(p)$ and the
right-hand side by $r(p)$, we observe that we can write $l(p)$ and
$r(p)$ as
\begin{align} \label{eq:l(p)}
l(p)&=\prod_{k=0}^{\frac{p-5}{4}}
\left(\frac{3}{4}+\frac{p-5}{4}-k\right)\left(\frac{3}{4}+\frac{p-1}{4}+k\right)\\
&\equiv \prod_{k=0}^{\frac{p-5}{4}}
\left(-\frac{1}{4}-k-k^2\right)\pmod{p^2} \notag
\end{align}
and
\begin{align} \label{eq:r(p)}
r(p)&=\prod_{k=0}^{\frac{p-5}{4}}
\left(\frac{1}{4}+\frac{p+1}{2}+\frac{p-5}{4}-k\right)\left(\frac{1}{4}+
\frac{p+1}{2}+\frac{p-1}{4}+k\right)\\
&\equiv \prod_{k=0}^{\frac{p-5}{4}}
\left(-\frac{1}{4}-k-k^2\right)\pmod{p^2}, \notag
\end{align}
which implies \eqref{eq2.3}. This completes the proof of
\eqref{eq2.2}.
\end{proof}

As a consequence, we have the following congruence.

\begin{Corollary}
Let $p$ be a prime such that $p\equiv 1\pmod{4}$. Then
\begin{equation}\label{eq2.4}
\prod_{\substack{k=0 \\ k\neq
\frac{3p-3}{4}}}^{p-1}(3+4k) \equiv \prod_{\substack{k=0\\
k\neq\frac{p-1}{4}}}^{p-1} (1+4k) \pmod{p^2}.
\end{equation}
\end{Corollary}

\section{Simple properties of $s_n$ and the congruence \eqref{eq1.6} for $n=1$}

We first  observe that
\begin{equation}\label{eq3.1} s_n=\frac{\left( \frac{1}{4}\right)
_{n}^2 64^n}{ (n!)^2}= \frac{4^n}{(n!)^2} \prod_{i=0}^{n-1}
(1+4i)^2.\end{equation}

\begin{Lemma}\label{L2.1}
If $p$ is a prime satisfying $p\equiv 1\pmod{4}$, then
\begin{equation}\label{eq3.2}
    s_p\equiv s_1\pmod{p^2}.
\end{equation}
\end{Lemma}

\begin{proof}
From \eqref{eq3.1}, we find that
   $$s_p=\frac{4^p}{(p!)^2}\prod_{i=0}^{p-1}(1+4i)^2.$$
   Observe
that
$$s_p = \frac{4^p}{((p-1)!)^2}
\prod_{\substack{i=0 \\ i\neq \frac{p-1}{4}}}^{p-1} (1+4i)^2.$$ By
\eqref{eq2.4}, we find that
\begin{align*}
s_p& \equiv \frac{4^p}{((p-1)!)^2}
\prod_{\substack{i=0 \\
i\neq \frac{p-1}{4}}}^{p-1}(1+4i) \prod_{\substack{k=0 \\
k\neq \frac{3p-3}{4}}}^{p-1}(3+4k) \pmod{p^2}
   \\
&\equiv \frac{1}{3}\frac{4^p}{(p!)^2} \prod_{i=0}^{p-1}
(1+4i)\prod_{i=0}^{p-1}(3+4i) \pmod{p^2}.\end{align*} Therefore,
\begin{align}\label{eq3.3}
s_p &\equiv \frac{1}{3}\frac{4^p}{(p!)^2}
\prod_{i=0}^{p-1}(1+4i)(3+4i)\pmod{p^2}
\\
&\equiv
\frac{1}{3}\frac{4^p}{(p!)^2}\prod_{i=0}^{p-1}\frac{(1+4i)(3+4i)(2+4i)(4+4i)}{2^2(1+2i)(2+2i)}\pmod{p^2}
\notag \\
&\equiv
\frac{1}{3}\bin{4p}{2p}\bin{2p}{p}\pmod{p^2}.\notag\end{align} It is
known  that (see \cite{Kaz}, respectively \cite[Theorem~4]{Bailey}),
for positive integers $a$ and $b$, with $a\geq b$, and primes $p>3$,
\begin{equation}\label{eq3.4}\bin{pa}{pb}\equiv \bin{a}{b}\pmod{p^3}.\end{equation}
Using \eqref{eq3.4} in the last expression in  \eqref{eq3.3}, we
conclude that
$$s_p\equiv \frac{1}{3}\bin{4p}{2p}\bin{2p}{p}\equiv \frac{1}{3}\bin{4}{2}\bin{2}{1}\equiv 4\pmod{p^2}.$$
\end{proof}


We end this section with a simple observation. Let
\begin{equation}\label{eq3.5} F(n) = 4^{p-1}\prod_{\substack{j=0 \\
j\neq \frac{p-1}{4}}}^{p-1} \left(1+4j+4np\right)^2
\prod_{i=0}^{p-2} \frac{1}{\left(1+i+np\right)^2 }.\end{equation}
From \eqref{eq3.2}, we have the following congruence for $F(0)$.

\begin{Corollary}
\begin{equation}\label{eq3.6}
F(0)\equiv 1\pmod{p^2}.
\end{equation}
\end{Corollary}


\section{Completion of the proof of \eqref{eq1.6}}

\begin{Lemma} Let $F(n)$ be defined as in \eqref{eq3.5} and suppose $p\equiv 1\pmod{4}$.
Then $F(n)\pmod{p^2}$ is independent of $n$.
\end{Lemma}

\begin{proof}%
We first consider the denominator of $F(n)$. We have
\begin{align*} \prod_{i=0}^{p-2} \frac{1}{\left(1+i+np\right)^2 }
&=\prod_{k=1}^{(p-1)/2} \frac{1}{(np+k)^2((n+1)p-k)^2}\\
&\equiv\prod_{k=1}^{(p-1)/2}\frac{1}{k^2(p-k)^2}\pmod{p^2}.
\end{align*}

Next, we split the numerator of $F(n)$ into two parts, namely,
$$\prod_{\substack{j=0 \\ j\neq
\frac{p-1}{4}}}^{p-1} \left(1+4j+4np\right)^2 =A(n)B(n),$$ where
\begin{align*}A(n)&=\prod_{j=1}^{(p-1)/4}\left(1+4\left(\frac{p-1}{4}-j\right)+4np\right)^2
\\ &\qquad\qquad\times \left(1+4\left(\frac{p-1}{4}+j\right)+4np\right)^2\\
&\equiv \prod_{j=1}^{(p-1)/4} 16^2j^4\pmod{p^2}
\end{align*}
and
\begin{align*}B(n)&=\prod_{k=1}^{(p-1)/4} \left(4np+2p+3+4\left(\frac{p-1}{4}-k\right)\right)^2\\
&\qquad\qquad\times
\left(4np+2p+3+4\left(\frac{p-1}{4}+k-1\right)\right)^2
\\
&=\prod_{k=1}^{(p-1)/4}\left(4np+3p-(4k-2)\right)^2\left(4np+3p+(4k-2)\right)^2
\\
&\equiv\prod_{k=1}^{(p-1)/4} \left(-4+16k-16k^2\right)\pmod{p^2}.
\end{align*}
The above computations show that both $A(n)\pmod{p^2}$ and
$B(n)\pmod{p^2}$ are independent of $n$. Hence, $F(n)\pmod{p^2}$ is
independent of $n$.
\end{proof}

Using \eqref{eq3.6}, we arrive at the following conclusion.

\begin{Corollary} \label{C4.1} For all positive integers $n$ and
$p\equiv 1\pmod{p^2}$, we have
$$F(n)\equiv F(0)\equiv 1\pmod{p^2}.$$
\end{Corollary}

\begin{proof}[Completion of the proof of \eqref{eq1.6}]
Our aim is to show that $$s_{np}\equiv s_n\pmod{p^2}$$ for all
positive integers $n$ and primes $p\equiv 1\pmod{4}$. We shall
accomplish this by an induction on $n$.

From \eqref{eq3.1}, we find that
\begin{equation} \label{eq4.1}
s_{n+1}=4\left(\frac{1+4n}{1+n}\right)^2s_n.
\end{equation}
Therefore
\[
   s_{n+k}=  4^k \prod_{i=0}^{k-1}\left(\frac{1+4(i+n)}{1+n+i}\right)^2s_n.
\]
In particular,
\begin{equation} \label{eq4.2}
   s_{n+p}=  4^p \prod_{i=0}^{p-1}\left(\frac{1+4(i+n)}{1+n+i}\right)^2s_n.
\end{equation}

Now, for the induction hypothesis, suppose that
\begin{equation}\label{eq4.3}s_{np}\equiv
s_n\pmod{p^2}.\end{equation} By \eqref{eq4.2}, we find that
\begin{align*}s_{(n+1)p}&=s_{np+p}=s_{np}4^p\prod_{i=0}^{p-1}\left(\frac{1+4(i+np)}{1+i+np}\right)^2\\
& \equiv s_n
4^p\prod_{i=0}^{p-1}\left(\frac{1+4(i+np)}{1+i+np}\right)^2\pmod{p^2},\end{align*}
where we used \eqref{eq4.3} in the last congruence. We observe that,
if
\begin{equation}\label{eq4.4} 4^p\prod_{i=0}^{p-1}\left(\frac{1+4(i+np)}{1+i+np}\right)^2\equiv 4\left(\frac{1+4n}{1+n}\right)^2\pmod{p^2},\end{equation} then we would have
$$s_{(n+1)p}\equiv s_n 4\left(\frac{1+4n}{1+n}\right)^2 \equiv s_{n+1}\pmod{p^2},$$ by \eqref{eq4.1}.
But the congruence \eqref{eq4.4} is exactly the congruence in
Corollary~\ref{C4.1}. This completes our proof of \eqref{eq1.6}.
\end{proof}

\section{A Lemma for the proof of \eqref{eq1.6.1}}

\begin{Lemma}\label{L5.1}
Let $p=6q+1$ be a prime. Then
$$4^p\left(\frac{1}{6}\right)_p\equiv
\left(\frac{2}{3}\right)_p\pmod{p^3}.$$
\end{Lemma}

\begin{proof}
We want to reduce the congruence to one that we can manage. Clearing
denominators and dividing the terms which are multiples of $p$ on
both sides, we see that we need to prove that
\begin{align*} 2^{6q}1\cdot
7\cdots &(6q-5)(6q+7)\cdots (36q+1)\\ &\equiv 2\cdot 5\cdots
(12q-1)(12q+5)\cdots (18q+2) \pmod{p^2}.\end{align*} We next match
the terms $6q+1-6k$ to $6q+1+6k$ for $1\leq k\leq q$ and simplify
the left-hand side to
$$2^{6q}\prod_{k=1}^q (6q+1-6k)(6q+1+6k) \cdot M(q)\equiv
2^{8q}\cdot 3^{2q}\prod_{k=1}^q (-k^2)\cdot M(q)\pmod{p^2},$$ where
$$M(q) = \prod_{k=1}^{4q} (12q+1+6k).$$
But $M(q)$ can also be expressed as
$$M(q) =\prod_{k=1}^{2q} (24q+4-(6k-3))(24q+4+6k-3)
\equiv 3^{4q}\prod_{k=1}^{2q} (2k-1)^2 \pmod{p^2}.$$ Hence the
left-hand side is
$$2^{8q}\cdot 3^{6q}\prod_{k=1}^q (-k^2)\prod_{k=1}^{2q} (2k-1)^2.$$
Similarly, the right-hand side can be expressed as
$$\prod_{k=1}^{2q}(12q+2-3k)(12q+2+3k)\cdot N(q) \equiv
3^{4q}\prod_{k=1}^{2q} k^2 \cdot N(q)\pmod{p^2},$$ where $N(q)$ is
given by
\begin{align*}N(q)  &= \prod_{k=1}^q
\left(3q+\frac{1}{2}-\frac{6k-3}{2}\right)
\left(3q+\frac{1}{2}+\frac{6k-3}{2}\right)
\\ &\equiv \left(\frac{3}{2}\right)^{2q}\prod_{k=1}^q
(-(2k-1)^2)\pmod{p^2}.\end{align*} Simplifying both sides, we
observe that we need to prove that
$$2^{10q}\prod_{k=q+1}^{2q} (2k-1)^2 \equiv \prod_{k=q+1}^{2q} k^2 \pmod{p^2}.$$
We rewrite both sides, so that the above congruence turns out to be
equivalent to
$$2^{10q} \prod_{k=q+1}^{2q}(p-(2k-1))(p+(2k-1))\equiv \prod_{k=q+1}^{2q}(p-k)(p+k)\pmod{p^2}.$$
This leads to
$$2^{11q}\prod_{k=q+1}^{2q} (p-(2k-1))\equiv \prod_{k=q+1}^{2q} (p+k)\pmod{p^2},$$
since
$$\prod_{k=q+1}^{2q} (p+(2k-1))=2^{q}\prod_{k=q+1}^{2q} (p-k).$$
Now rewriting
$$\prod_{k=q+1}^{2q} (p-(2k-1)) = 2^q \prod_{k=q+1}^{2q} (3q-k+1),$$
we see that we must show that
\begin{multline*}
2^{12q}(q+1)(q+2)\cdots (2q)\\
\equiv (q+1)(q+2)\cdots (2q)
\left(1+p\left(H_{2q}-H_q\right)\right)\pmod{p^2},
\end{multline*}
where
$$H_n=\sum_{k=1}^{n} \frac{1}{k}.$$
Equivalently, we need to verify that
$$\frac{2^{6q}-1}{p}\cdot 2 \equiv H_{2q}-H_q \pmod{p}.$$
But it is known (see \cite[Theorem~132]{Hardy-Wright}) that
$$\frac{2^{p-1}-1}{p}\equiv H_{6q}-\frac{H_{3q}}{2}\pmod{p}.$$
Since $$H_{6q}\equiv 0\pmod{p},$$ it suffices to show that
$$-H_{3q}+H_q-H_{2q}\equiv 0\pmod{p}.$$
Observe that $$H_{3q}=1+\frac{1}{2}+\cdots
+\frac{1}{q}+H_{2q}-H_q+\frac{1}{2q+1}+\cdots +\frac{1}{3q}.$$ Now,
for $1\leq i\leq q$, we pair the terms in the sums at both ends as
follows:
$$\frac{1}{i}+\frac{1}{3q+1-i}=\frac{3q+1}{i(3q+1-i)}\equiv
\frac{1}{2i(3q+1-i)}\equiv \frac{1}{i}-\frac{2}{2i-1}\pmod{p}.$$
Hence, we deduce that
$$H_{3q}\equiv H_q-2\left(H_{2q}-\frac{H_q}{2}\right)+H_{2q}-H_q
\equiv -H_{2q}+H_q\pmod{p},$$ which completes the proof of the
lemma.
\end{proof}

We are now ready to show that if $$t_n = 108^n
\frac{\left(\frac{1}{6}\right)_n\left(\frac{1}{3}\right)_n}{(1)_n^2}$$
then \begin{equation}\label{eq5.1} t_p\equiv
t_1\pmod{p^2}\end{equation} for all primes $p\equiv 1\pmod{6}.$ By
Lemma~\ref{L5.1},
\begin{equation}\label{eq5.0.5} t_p\equiv
27^p\frac{\left(\frac{2}{3}\right)_p\left(\frac{1}{3}\right)_p}{(1)_p^2}\pmod{p^2}.\end{equation}
But the last expression can be written as
$$\binom{3p}{p}\binom{2p}{p}\equiv 6\equiv t_1\pmod{p^2},$$
by using \eqref{eq3.4}. This completes the proof of \eqref{eq5.1}.

The proof of \eqref{eq1.6.1} for $n>1$ is similar to the proof of
\eqref{eq1.6}. We will simply list the corresponding identities that
are needed in the proof. These are:

\begin{enumerate}[(i)]

\item The sequence $t_n$ satisfies
$$t_{n+1}=6\frac{(1+6n)(1+3n)}{(1+n)^2}t_n$$
and $$t_{n+p} = t_n 6^p\prod_{i=0}^{p-1}
\frac{(1+6n+6i)(1+3n+3i)}{(1+n+i)^2}.$$
\item The expression $$G(n) = 6^{p-1}\prod_{\substack{j=0 \\ j\neq
\frac{p-1}{6}}}^{p-1}(1+6j+6np)\prod_{\substack{j=0 \\ j\neq
\frac{p-1}{3}}}^{p-1}(1+3j+3np)\prod_{i=0}^{p-2}\frac{1}{(1+i+np)^2}$$
is independent of $n$ modulo $p^2$, and $$G(n)\equiv G(0)\equiv
1\pmod{p^2}.$$

%

\end{enumerate}

The proofs of (i) and (ii) are similar to those presented in
Section~4.

\medskip\noindent \emph{Acknowledgments}.  The first author was
supported by NUS Academic Research Grant R-146-000-103-112. The work
was carried out when the first author was visiting the
Max-Planck-Institut f\"ur Mathematik (MPIM). He thanks the MPIM for
providing a nice research environment. He also likes to take the
opportunity to thank Elisavet Konstantinou for inviting him to the
University of the Aegean, where he met the second author and had
many fruitful discussions.

The third author was partially supported by the Austrian Science
Foundation FWF, grants Z130-N13 and S9607-N13, the latter in the
framework of the National Research Network ``Analytic Combinatorics
and Probabilistic Number Theory."

The fourth author was partially supported by Science Foundation
Ireland 08/RFP/MTH1081.

\end{document}